\documentclass[a4paper,12pt, reqno]{amsart}

\usepackage[T1]{fontenc}
\PassOptionsToPackage{no-math}{fontspec}
\usepackage{iftex}

\ifPDFTeX
\usepackage[finemath]{kotex}
\usepackage{dhucs-nanumfont}
\fi
\ifXeTeX
\usepackage[unfonts]{kotex}
\xetexkofontregime{latin}[puncts=prevfont, colons=prevfont, cjksymbols=hangul]

\defaultfontfeatures+{
BoldFeatures={FakeBold=2},
ItalicFeatures={FakeSlant=0.17},
SlantedFeatures={FakeSlant=0.17},
BoldItalicFeatures={FakeBold=2,FakeSlant=0.17},
BoldSlantedFeatures={FakeBold=2,FakeSlant=0.17},
}

\fi
\usepackage{a4wide}

\usepackage{subfigure}

\usepackage{graphicx}

\usepackage{color}
\usepackage{ifpdf}
\ifpdf %if using pdfLaTeX in PDF mode
\DeclareGraphicsExtensions{.pdf,.png,.mps}
\usepackage{pgf}
\else %if using LaTeX or pdfLaTeX in DVI mode
\usepackage{graphicx}
\DeclareGraphicsExtensions{.eps,.bmp}
\usepackage{pgf}
\fi
\usepackage{epic}
\usepackage{wrapfig}% package for wrapfigure environment

\usepackage{dsfont}
\usepackage{amsmath, amssymb}
\usepackage{multirow}
\usepackage{tikz}
\usetikzlibrary{cd}
\usepackage{enumerate}
\usepackage[%
]{todonotes}

\theoremstyle{plain}
\newtheorem{thm}{Theorem}%[section]
\newtheorem{cor}[thm]{Corollary}

\newtheorem{conj}[thm]{Conjecture}

\theoremstyle{definition}

\newcommand{\rmk}{\par\noindent{\it Remark. }}

\usepackage{amsopn}

\newcommand\set[1]{\left\{#1\right\}}

\def\Z{\mathbb{Z}}

\def\P{\mathcal{P}}

\def\*{^{*}}
\def\k{^{(k)}}
\def\one{^{(1)}}
\def\two{^{(2)}}
\def\three{^{(3)}}

\def\nm{_{n,m}}
\def\nkn{_{n,kn}}
\def\n1m1{_{n-1,m-1}}
\newcommand\aug[1]{{#1}^{+}}
\newcommand\di[1]{{#1}^{-}}

\def\ptd{\pi}
\def\dtp{\delta}
\mathchardef\mhyphen="2D

\title%[]%
{On Delannoy paths without peaks and valleys}

\author{Seunghyun Seo}
\address[Seunghyun Seo]{Department of Mathematics Education, Kangwon National University, Chuncheon, 24341, South Korea}
\email{shyunseo@kangwon.ac.kr}

\author{Heesung Shin$^\dag$}
\address[Heesung Shin]{Department of Mathematics, Inha University, 100 Inharo, Michuhol, Incheon, 22212, South Korea}
\email{shin@inha.ac.kr}
\thanks{\dag Corresponding author}

\date{\today}
\begin{document}
\maketitle
\begin{abstract}
A lattice path is called \emph{Delannoy} if each of its steps belongs to $\left\{N, E, D\right\}$, where $N=(0,1)$, $E=(1,0)$, and $D=(1,1)$ steps. \emph{Peak}, \emph{valley}, and \emph{deep valley} are denoted by $NE$, $EN$, and $EENN$ on the lattice path, respectively.

In this paper, we find a bijection between $\mathcal{P}_{n,m}(NE, EN)$ and a specific subset of ${\mathcal{P}_{n,m}}(D, EENN)$, where $\mathcal{P}_{n,m}(NE, EN)$ is the set of Delannoy paths from the origin to $(n,m)$ without peaks and valleys, and ${\mathcal{P}_{n,m}}(D, EENN)$ is the set of Delannoy lattice paths from the origin to $(n,m)$ without diagonal steps and deep valleys. We also enumerate the number of Delannoy paths without peaks and valleys on the restricted region $\left\{ (x,y) \in \mathbb{Z}^2 : y \ge k x \right\}$ for a positive integer $k$.
\end{abstract}

%\todototoc
%\listoftodos

%\tableofcontents

%%%%%%%%%%%%%%%%
\section{Introduction}
For two lattice points $A$ and $B$,
a \emph{lattice path} from $A$ to $B$
is a sequence of lattice points
$$(v_{1}, v_{2},\dots ,v_{k})$$
of $\Z^2$
with $v_1 = A$ and $v_{k} = B$.
For each $i=1, \dots, k-1$, a consecutive difference $s_i = v_{i+1} - v_{i}$ of $(v_{1}, v_{2},\dots ,v_{k})$
is called a \emph{step} of the lattice path.
Conventionally, a lattice path can be represented as a word
$ s_1 s_2 \dots s_{k-1}$
with starting point $v_1$.

A $(0,1)$ step is called a \emph{north} step, denoted by $N$;
a $(1,0)$ step is called an \emph{east} step, denoted by $E$;
a $(1,1)$ step is called a \emph{diagonal} step, denoted by $D$.
A lattice path is called \emph{Delannoy} if each of its steps belongs to $\set{N, E, D}.$
%Similarly, a lattice path is called \emph{North-East} if every step of the lattice path belongs to $\set{N, E}.$

Let $n$ and $m$ be nonnegative integers.
Let $\P\nm$ be the set of Delannoy paths from the origin to $(n,m)$.
It is well-known \cite[p. 81]{Com74} that 
$$
\#\P\nm
= \sum_{d=0}^n \binom{n+m-d}{n-d,m-d,d}
= \sum_{j=0}^{n} \binom{n}{j} \binom{m}{j} 2^j .
$$

If a pair $NE$ (resp. $EN$) appears consecutively on the lattice path,
it is called a \emph{peak} (resp. \emph{valley}).
%A \emph{sharp turn} means a peak $NE$ or valley $EN$ on a lattice path.
If a quadruple $EENN$ appears consecutively on the lattice path,
it is called a \emph{deep valley}.

The collection of Delannoy paths that avoid all specific patterns $\omega_1, \dots, \omega_k$ is denoted as
$$\P\nm(\omega_1, \dots, \omega_k) = \set{P \in \P\nm : \text{There are no patterns $\omega_1, \dots, \omega_k$ in $P$}}.$$

Let $\P\nm(NE, EN)$ be the set of Delannoy paths without peaks and valleys in $\P\nm$.
To calculate the number of reduced alignments between two DNA sequences in Bioinformatics,
Andrade et al. \cite{AANT14} found that the cardinality of $\P\nm(NE, EN)$ is
\begin{align}
\sum_{i\ge0} (-1)^i \left[ \binom{n+m-3i}{n -2i, m -2i, i} - \binom{(n-1)+(m-1)-3i}{(n-1)-2i, (m-1)-2i,i} \right],
\label{eq:s_gen}
\end{align}
which is given by entry A047080 in the OEIS \cite{Slo18}.

In this paper, 
given a lattice path $P = (v_1, v_2, \dots, v_k)$,
with $v_0$ and $v_{k+1}$ satisfying
\begin{align*}
v_1-v_{0} = E\quad\text{and}\quad v_{k+1}-v_k = N,
\end{align*}
$\aug{P} = (v_0, v_1, \dots, v_{k+1})$ is called an \emph{augmented} path of $P$, and
$\di{P} = (v_2, \dots, v_{k-1})$ is called a \emph{diminished} path of $P$.
Given a set $S$ of lattice paths, $\aug{S}$ and $\di{S}$ are defined by
$$\aug{S} = \set{\aug{P} : P \in S} \quad\text{and}\quad \di{S} = \set{\di{P} : P \in S}.$$

%For two nonnegative $n$ and $m$,
%let $\P\nm(NE, EN)$ be the set of
%Delannoy paths from the origin to $(n,m)$ without peaks and valleys.
%Similarly, l
Let $\di{\aug{\P\nm}(D, EENN)}$ be the set of
Delannoy paths from the origin to $(n,m)$ whose augmented path does not include diagonal steps or deep valleys.

We construct two bijections 
$\dtp$ from $\P\nm(NE, EN)$ to $\di{\aug{\P\nm}(D, EENN)}$
and
$\ptd$ from $\di{\aug{\P\nm}(D, EENN)}$ to $\P\nm(NE, EN)$.
$$
\begin{tikzcd}
\P\nm&
\P\nm\\
\P\nm(NE, EN) 
\arrow[u, no head, "\cup"] \arrow[r, shift left=.7, "\dtp"] &
\di{\aug{\P\nm}(D, EENN)} 
\arrow[u, no head, "\cup"] \arrow[l, shift left=.7, "\ptd" ]
\end{tikzcd}
$$

Let $k$ be a positive integer.
Let $\P\nkn\k$ be the set of Delannoy paths $P \in \P\nkn$ on the half-plane $y \geq kx$.
Song \cite{Son05} introduced a \emph{$k$-Schr\"oder path} $P$ of size $n$
if $P \in \P\nkn\k$ for nonnegative integers $n$,
where the cardinality of $\P\nkn\k$ equals 
$$
\#\P\nkn\k
= \sum_{d=0}^n \frac{1}{kn-d+1}\binom{kn+n-d}{kn-d,n-d,d}
= \frac{1}{n}\sum_{j=1}^{n} \binom{kn}{j-1} \binom{n}{j} 2^j.
$$

It is known \cite[A086581]{Slo18}
that the number of Delannoy paths without diagonal steps and deep valleys from $(0,0)$ to $(n,n)$ 
on the half-plane $y\geq x$ is as follows:
\begin{align*}
\#\P\one_{n,n}(D, EENN) &= \sum_{m\ge0} \frac{1}{m+1} {2m \choose m}{n+m \choose 3m}.
\end{align*}

In this paper, 
we prove that
$\P\k\nkn(NE, EN)$ 
and 
$\P\k\nkn(D, EENN)$ 
are equinumerous via two bijections $\dtp$ and $\ptd$.
$$
\begin{tikzcd}
%\P\nm&
%\P\nm\\
%\P\nm(D) \arrow[r, shift left=.7, "\ptd"] \arrow[u, no head, "\cup"] &
%\P\nm(NE) \arrow[l, shift left=.7, "\dtp"] \arrow[u, no head, "\cup"] \\
\P\nkn(NE, EN) \arrow[r, shift left=.7, "\dtp"] & % \arrow[u, no head, "\cup"] &
\di{\aug{\P\nkn}(D, EENN)} \arrow[l, shift left=.7, "\ptd" ] \\ %\arrow[u, no head, "\cup"] \\
\P\k\nkn(NE,EN) \arrow[u, no head, "\cup"] \arrow[r, shift left=.7, "\dtp"] &
\P\k\nkn(D, EENN) \arrow[u, no head, "\cup"] \arrow[l, shift left=.7, "\ptd" ]
\end{tikzcd}
$$

For $k=1$ or $2$,
we find the numbers of Delannoy paths without peaks and valleys from $(0,0)$ to $(n,kn)$ 
on the half-plane $y\geq kx$ as follows:
\begin{align*}
\#\P\one_{n,n} (NE, EN) &= \sum_{m\ge0} \frac{1}{m+1} {2m \choose m}{n+m \choose 3m},\\
\#\P\two_{n,2n} (NE, EN) &= \sum_{m\ge0} \frac{ (-1)^{n-m}}{m+1} {3m+1 \choose m}{m+1 \choose n-m}.
\end{align*}

The remainder of this paper is organized as follows.
In Section~\ref{sec:delannoy},
we review previous studies and analyze the number of Delannoy paths without peaks and valleys.
In Section~\ref{sec:map},
we construct two bijections $\dtp$ and $\ptd$ 
between $\P\nm(NE, EN)$ and $\di{\aug{\P\nm}(D, EENN)}$.
In Section~\ref{sec:path},
we address Delannoy paths without peaks and valleys on the half-plane $y\geq kx$.
In Section~\ref{sec:without},
we enumerate the cardinalities of $\P\one_{n,n} (NE, EN)$ and $\P\two_{n,2n} (NE, EN)$.
We also present 
Conjecture~\ref{conj:1}, which is about 
Catalan paths avoiding symmetric peaks \cite{Eli21},
and
Conjecture~\ref{conj:2}, which is about 
inversion sequences avoiding the pattern $102$ \cite{MS15}.

%\begin{conj}
%For a positive integer $n$,
%there exists a bijection between the following two sets:
%\begin{enumerate}[(i)]
%\item the set $E\P\one_{n, n}(NE, EN)$ of Delannoy paths from $(0,0)$ to $(n, n)$
%consisting of north-steps $N=(0,1)$, east-steps $E=(1,0)$, and diagonal-steps $D=(1,1)$
%that 
%avoid patterns $NE$ and $EN$, 
%end with east step $E$, 
%and are not below $y=x$ 
%\item the set of Catalan paths of size $n+1$ avoiding symmetric peaks.
%\end{enumerate}
%\end{conj}
%
%\begin{conj}
%For a positive integer $n$,
%there exists a bijection between the following two sets:
%\begin{enumerate}[(i)]
%%\item the set $\P\two_{n-1, 2n-1}(NE, EN)$ of Delannoy paths from $(0,0)$ to $(n-1, 2n-1)$
%%consisting of north-steps $N=(0,1)$, east-steps $E=(1,0)$, and diagonal-steps $D=(1,1)$ avoiding the patterns $NE$ and $EN$, that do not go below $y=2x$ and
%\item the set $D\P\two_{n, 2n}(NE, EN)$ of Delannoy paths from $(0,0)$ to $(n, 2n)$
%consisting of north-steps $N=(0,1)$, east-steps $E=(1,0)$, and diagonal-steps $D=(1,1)$ 
%that 
%avoid the patterns $NE$ and $EN$, 
%end with diagonal step $D$, 
%and are not below $y=2x$ and
%\item the set $\mathcal{IS}_n(102)$ of inversion sequences of length $n$ avoiding the pattern $102$.
%\end{enumerate}
%\end{conj}

\section{Delannoy Paths}
\label{sec:delannoy}
%\begin{defn}[Peak, Valley, Deep Valley]
%A pattern $NE$ is called a \emph{peak}.
%A pattern $EN$ is called a \emph{valley}.
%A pattern $EENN$ is called a \emph{deep valley}.
%\end{defn}
%\begin{defn}[Lattice path]
%\shin{$\Z^2$가 영역 $R$로 변경되었을 때의 정의는 후술하자.}
%A \emph{lattice path} $P$ in $\Z^2$ with steps in $S$ is a sequence $(v_{1}, v_{2},\dots ,v_{k})$ of $\Z^2$
%such that each consecutive difference $s_i = v_{i+1}-v_{i}$ lies in $S$.
%\end{defn}
%\subsection*{Delannoy path}
%A \emph{Delannoy path} is a lattice path in $\Z^2$ with steps
%in $\set{N, E, D}$.

Let $n$ and $m$ be nonnegative integers.
Let $h(n,m)$ be the cardinality of $\P\nm(NE, EN)$.
Andrade et al. \cite{AANT14} mentioned that 
$h(n,m)$
satisfies the recurrence
\begin{align}
h(n,m) = h(n-1, m) + h(n, m-1) - h(n-2, m-2)
\label{eq:s_rec}
\end{align}
for $n \ge2$ and $m\ge 2$.
With initial conditions $h(n,0)=h(0,m)=1$, 
they deduced the generating function of $h(n,m)$ and 
found Formula \eqref{eq:s_gen} for $h(n,m)$.

Because their process of finding \eqref{eq:s_gen} was algebraic,
we want to find a combinatorial interpretation of $h(n,m)$.
It is natural to
consider the set $\P\nm(D, EENN)$ of Delannoy paths $P$ in $\P\nm$
without diagonal steps and deep valleys
as a combinatorial object satisfying the recurrence \eqref{eq:s_rec},
letting $b(n,m)$ be the cardinality of $\P\nm(D, EENN)$.
By considering the last pattern of the paths, $b(n,m)$ also satisfies
%\begin{align}
%b(n,m) =
%\begin{cases}
%b(n-1, m) + b(n, m-1) \\
%\qquad\qquad\qquad - b(n-2, m-2)& \text{if $(n,m) \neq (0,0)$}\\
%1 & \text{if $(n,m) = (0,0)$}
%\end{cases}
%\label{eq:recb}
%\end{align}
\begin{align}
b(n,m) =
\begin{cases}
b(n-1, m) + b(n, m-1) - b(n-2, m-2)& \text{if $(n,m) \neq (0,0)$,}\\
1 & \text{if $(n,m) = (0,0)$}.
\end{cases}
\label{eq:recb}
\end{align}
Conventionally, $b(n,m) = 0$ for $n<0$ or $m<0$.
Using the principle of inclusion and exclusion,
we obtain
\begin{align}
b(n,m) = \sum_{i\ge0} (-1)^i \binom{n+m-3i}{n -2i, m -2i, i},
\label{eq:b_gen}
\end{align}
where $i$ keeps track the number of deep valleys, which is not in the OEIS \cite{Slo18}.

Comparing \eqref{eq:s_gen} with \eqref{eq:b_gen}, we observe that
\begin{align}
h(n,m) = b(n,m) - b(n-1,m-1).
\label{eq:diff}
\end{align}

%When $n=m$
%$$
%\#\P_{n,n}(D, EENN)= \sum_{i\ge0} (-1)^i \binom{2n-3i}{n -2i, n -2i, i}
%$$
%$$
%\sum_{k=0}^{n} \sum_{i=0}^{k} \binom{n}{i}^2 \binom{n+k-i}{ k-i}
%$$
%is given by entry A108626 in the OEIS \cite{Slo18}.

Now, we introduce another combinatorial object whose cardinality is the right-hand side of \eqref{eq:diff}.
Recall that $\di{\aug{\P\nm}(D, EENN)}$ is the set of Delannoy paths $P$ in $\P\nm$,
where the augmented lattice path $\aug{P}$ does not have a diagonal step $D$ or a deep valley $EENN$.
Let $a(n,m)$ be the cardinality of $\di{\aug{\P\nm}(D, EENN)}$.
We show a combinatorial proof of
\begin{align}
a(n,m) = b(n,m) - b(n-1,m-1).
\label{eq:tau}
\end{align}

%We can explain the recurrence \eqref{eq:tau} combinatorially.
%Recall that $\di{\aug{\P\nm}(D, EENN)}$ be the set of
%North-East lattice paths from the origin to the points $(n,m)$ whose augmented path does not include deep valleys.
%Simply, denote by $\P\nm(D, EENN)$ the set of North-East lattice paths from the origin to the points $(n,m)$ without deep valleys.
%Note that $\di{\aug{\P\nm}(D, EENN)}$ is a subset of $\P\nm(D, EENN)$.
Consider a mapping
%$\tau$ from $\P\nm(D, EENN) \setminus \di{\aug{\P\nm}(D, EENN)}$ to $\P\n1m1(D, EENN)$
$$\tau : \P\nm(D, EENN) \setminus \di{\aug{\P\nm}(D, EENN)} \to \P\n1m1(D, EENN)$$
with two rules:
\begin{align*}
\alpha EEN &\mapsto E \alpha, \\
ENN \beta &\mapsto N \beta.
\end{align*}
According to these rules, 
for a given $ENN \gamma EEN$ in $\P\nm(D, EENN) \setminus \di{\aug{\P\nm}(D, EENN)}$,
we have two candidates $EENN \gamma$ and $N \gamma EEN$ as $\tau(ENN \gamma EEN)$.
Because the first candidate $EENN \gamma$ includes a deep valley, it does not belong to $\P\n1m1(D, EENN)$.
As a result, $ENN \gamma EEN$ should correspond to $N \gamma EEN$ in $\P\n1m1(D, EENN)$.

In addition, the inverse mapping of $\tau$ is well-defined, and $\tau$ is bijective.
Therefore, we immediately obtain the equation
$$
b(n,m)-a(n,m)=b(n-1,m-1),
$$
which is equivalent to \eqref{eq:tau}.

%Let $b(n,m)$ be the cardinality of $\P\nm(D, EENN)$ for two nonnegative integers $n$ and $m$.
%Conventionally, $b(n,m)=0$ if either $n$ or $m$ is negative integer.
%From the definition of $\P\nm(D, EENN)$, the number $b(n,m)$ satisfy a recurrence relation
%\begin{align}
%b(n,m) =
%\begin{cases}
%b(n-1, m) + b(n, m-1) \\
%\qquad\qquad\qquad - b(n-2, m-2)& \text{if $(n,m) \neq (0,0)$}\\
%1 & \text{if $(n,m) = (0,0)$}
%\end{cases}
%\label{eq:recb}
%\end{align}
%Thus the generating function $F_B(x,y)$ of $b(n,m)$ is
%$$F_B(x,y) = \sum_{n, m\ge 0} b(n,m) x^n y^m = \frac{1}{1-x-y+x^2y^2}.$$

%Let $a(n,m)$ be the cardinality of $\di{\aug{\P\nm}(D, EENN)}$ for two nonnegative integers $n$ and $m$, and
\rmk
The generating function $F_B(x,y)$ of $b(n,m)$ is induced from \eqref{eq:recb} as follows:
\begin{align}
F_B(x,y) = \sum_{n, m\ge 0} b(n,m) x^n y^m = \frac{1}{1-x-y+x^2y^2}.
\label{eq:FB}
\end{align}
When the generating function of $a(n,m)$ is defined by
$$F_A(x,y) = \sum_{n, m\ge 0} a(n,m) x^n y^m,$$
Equation~\eqref{eq:tau} yields
$$F_B(x,y) - F_A(x,y) = xy F_B(x,y).$$
Thus, we have the generating function of $a(n,m)$
$$F_A(x,y) = \frac{1-xy}{1-x-y+x^2y^2},$$
which is also the generating function of $h(n,m)$.

%\rmk In fact, by the peak-to-diagonal mapping $\ptd$ and diagonal-to-peak mapping $\dtp$, for two nonnegative integers $n$ and $m$,
%$$h(n,m) = a(n,m)$$
%and by the bijection $\tau$, for two positive integers $n$ and $m$,
%$$b(n,m) - a(n,m) = b(n-1,m-1).$$
%Thus, \eqref{eq:recb} implies to \eqref{eq:s_rec}, that is,
%for two integers $n \geq 2$ and $m \geq 2$,
%\begin{align*}
%h(n,m)
%=& b(n,m) - b(n-1,m-1)\\
%=& \left( b(n-1,m) - b(n,m-1) + b(n-2,m-2) \right)\\
% &+\left( b(n-2,m-1) - b(n-1,m-2) + b(n-3,m-3)\right)\\
%=& h(n-1, m) + h(n, m-1) - h(n-2, m-2).
%\end{align*}

\section{Two Mappings}
\label{sec:map}

%\subsection*{North-East lattice path}
A \emph{North-East lattice path} is a lattice path
in $\Z^2$ with steps in $\set{N, E}$.
Let $\P\nm(D)$ be the set of North-East lattice paths from the origin to $(n,m)$.
%for nonnegative integers $n$ and $m$.
Let $\P\nm(NE)$ be the set of Delannoy paths without peaks in $\P\nm$.

Given $P \in \P\nm(D)$,
a mapping $\ptd$ changes each peak of $P$ to a diagonal step.
Evidently, 
$\pi$ is well-defined 
and
$\ptd(P) \in \P\nm(NE)$.
Considering a reverse of $\ptd$,
%given a Delannoy path without peaks,
%it is also well-defined to change every diagonal steps to peak,
%called by the \emph{diagonal-to-peak} mapping$\dtp$.
given $Q \in \P\nm(NE)$,
%it is well-defined to change each diagonal step of $Q$ to a peak
%called by the \emph{diagonal-to-peak} mapping $\dtp$. 
%Evidently, $\dtp(Q) \in \P\nm(D)$.
a mapping $\dtp$ changes each diagonal step of $Q$ to a peak.
Evidently, 
$\dtp$ is well-defined 
and 
$\dtp(Q) \in \P\nm(D)$.

%Obviously,
%the peak-to-diagonal mapping $\ptd$ and the diagonal-to-peak mapping $\dtp$ keep the starting point and the ending point of two corresponding paths and

\begin{figure}
%$$\input{mapping.TpX}$$
$$
\centering
\begin{pgfpicture}{-8.00mm}{-8.00mm}{134.00mm}{50.00mm}
\pgfsetxvec{\pgfpoint{0.60mm}{0mm}}
\pgfsetyvec{\pgfpoint{0mm}{0.60mm}}
\color[rgb]{0,0,0}\pgfsetlinewidth{0.30mm}\pgfsetdash{}{0mm}
\pgfsetlinewidth{0.60mm}\pgfmoveto{\pgfxy(0.00,0.00)}\pgflineto{\pgfxy(10.00,0.00)}\pgflineto{\pgfxy(10.00,10.00)}\pgflineto{\pgfxy(10.00,20.00)}\pgflineto{\pgfxy(20.00,20.00)}\pgflineto{\pgfxy(30.00,20.00)}\pgflineto{\pgfxy(30.00,30.00)}\pgflineto{\pgfxy(30.00,40.00)}\pgflineto{\pgfxy(30.00,50.00)}\pgflineto{\pgfxy(40.00,50.00)}\pgflineto{\pgfxy(40.00,60.00)}\pgflineto{\pgfxy(50.00,60.00)}\pgflineto{\pgfxy(60.00,60.00)}\pgflineto{\pgfxy(60.00,70.00)}\pgfstroke
\pgfcircle[fill]{\pgfxy(0.00,0.00)}{0.60mm}
\pgfsetlinewidth{0.30mm}\pgfcircle[stroke]{\pgfxy(0.00,0.00)}{0.60mm}
\pgfcircle[fill]{\pgfxy(10.00,0.00)}{0.60mm}
\pgfcircle[stroke]{\pgfxy(10.00,0.00)}{0.60mm}
\pgfcircle[fill]{\pgfxy(10.00,10.00)}{0.60mm}
\pgfcircle[stroke]{\pgfxy(10.00,10.00)}{0.60mm}
\pgfcircle[fill]{\pgfxy(10.00,20.00)}{0.60mm}
\pgfcircle[stroke]{\pgfxy(10.00,20.00)}{0.60mm}
\pgfcircle[fill]{\pgfxy(20.00,20.00)}{0.60mm}
\pgfcircle[stroke]{\pgfxy(20.00,20.00)}{0.60mm}
\pgfcircle[fill]{\pgfxy(30.00,20.00)}{0.60mm}
\pgfcircle[stroke]{\pgfxy(30.00,20.00)}{0.60mm}
\pgfcircle[fill]{\pgfxy(30.00,30.00)}{0.60mm}
\pgfcircle[stroke]{\pgfxy(30.00,30.00)}{0.60mm}
\pgfcircle[fill]{\pgfxy(30.00,40.00)}{0.60mm}
\pgfcircle[stroke]{\pgfxy(30.00,40.00)}{0.60mm}
\pgfcircle[fill]{\pgfxy(30.00,50.00)}{0.60mm}
\pgfcircle[stroke]{\pgfxy(30.00,50.00)}{0.60mm}
\pgfcircle[fill]{\pgfxy(40.00,50.00)}{0.60mm}
\pgfcircle[stroke]{\pgfxy(40.00,50.00)}{0.60mm}
\pgfcircle[fill]{\pgfxy(40.00,60.00)}{0.60mm}
\pgfcircle[stroke]{\pgfxy(40.00,60.00)}{0.60mm}
\pgfcircle[fill]{\pgfxy(50.00,60.00)}{0.60mm}
\pgfcircle[stroke]{\pgfxy(50.00,60.00)}{0.60mm}
\pgfcircle[fill]{\pgfxy(60.00,60.00)}{0.60mm}
\pgfcircle[stroke]{\pgfxy(60.00,60.00)}{0.60mm}
\pgfcircle[fill]{\pgfxy(60.00,70.00)}{0.60mm}
\pgfcircle[stroke]{\pgfxy(60.00,70.00)}{0.60mm}
\pgfmoveto{\pgfxy(-10.00,0.00)}\pgflineto{\pgfxy(-10.00,0.00)}\pgfstroke
\pgfmoveto{\pgfxy(-10.00,0.00)}\pgflineto{\pgfxy(80.00,0.00)}\pgfstroke
\pgfmoveto{\pgfxy(80.00,0.00)}\pgflineto{\pgfxy(77.20,0.70)}\pgflineto{\pgfxy(80.00,0.00)}\pgflineto{\pgfxy(77.20,-0.70)}\pgflineto{\pgfxy(80.00,0.00)}\pgfclosepath\pgffill
\pgfmoveto{\pgfxy(80.00,0.00)}\pgflineto{\pgfxy(77.20,0.70)}\pgflineto{\pgfxy(80.00,0.00)}\pgflineto{\pgfxy(77.20,-0.70)}\pgflineto{\pgfxy(80.00,0.00)}\pgfclosepath\pgfstroke
\pgfmoveto{\pgfxy(0.00,-10.00)}\pgflineto{\pgfxy(0.00,80.00)}\pgfstroke
\pgfmoveto{\pgfxy(0.00,80.00)}\pgflineto{\pgfxy(-0.70,77.20)}\pgflineto{\pgfxy(0.00,80.00)}\pgflineto{\pgfxy(0.70,77.20)}\pgflineto{\pgfxy(0.00,80.00)}\pgfclosepath\pgffill
\pgfmoveto{\pgfxy(0.00,80.00)}\pgflineto{\pgfxy(-0.70,77.20)}\pgflineto{\pgfxy(0.00,80.00)}\pgflineto{\pgfxy(0.70,77.20)}\pgflineto{\pgfxy(0.00,80.00)}\pgfclosepath\pgfstroke
\pgfsetlinewidth{0.60mm}\pgfmoveto{\pgfxy(140.00,0.00)}\pgflineto{\pgfxy(150.00,0.00)}\pgflineto{\pgfxy(150.00,10.00)}\pgflineto{\pgfxy(155.00,15.00)}\pgflineto{\pgfxy(160.00,20.00)}\pgflineto{\pgfxy(170.00,20.00)}\pgflineto{\pgfxy(170.00,30.00)}\pgflineto{\pgfxy(170.00,40.00)}\pgflineto{\pgfxy(175.00,45.00)}\pgflineto{\pgfxy(180.00,50.00)}\pgflineto{\pgfxy(185.00,55.00)}\pgflineto{\pgfxy(190.00,60.00)}\pgflineto{\pgfxy(200.00,60.00)}\pgflineto{\pgfxy(200.00,70.00)}\pgfstroke
\pgfcircle[fill]{\pgfxy(140.00,0.00)}{0.60mm}
\pgfsetlinewidth{0.30mm}\pgfcircle[stroke]{\pgfxy(140.00,0.00)}{0.60mm}
\pgfcircle[fill]{\pgfxy(150.00,0.00)}{0.60mm}
\pgfcircle[stroke]{\pgfxy(150.00,0.00)}{0.60mm}
\pgfcircle[fill]{\pgfxy(150.00,10.00)}{0.60mm}
\pgfcircle[stroke]{\pgfxy(150.00,10.00)}{0.60mm}
\pgfcircle[fill]{\pgfxy(160.00,20.00)}{0.60mm}
\pgfcircle[stroke]{\pgfxy(160.00,20.00)}{0.60mm}
\pgfcircle[fill]{\pgfxy(170.00,20.00)}{0.60mm}
\pgfcircle[stroke]{\pgfxy(170.00,20.00)}{0.60mm}
\pgfcircle[fill]{\pgfxy(170.00,30.00)}{0.60mm}
\pgfcircle[stroke]{\pgfxy(170.00,30.00)}{0.60mm}
\pgfcircle[fill]{\pgfxy(170.00,40.00)}{0.60mm}
\pgfcircle[stroke]{\pgfxy(170.00,40.00)}{0.60mm}
\pgfcircle[fill]{\pgfxy(180.00,50.00)}{0.60mm}
\pgfcircle[stroke]{\pgfxy(180.00,50.00)}{0.60mm}
\pgfcircle[fill]{\pgfxy(190.00,60.00)}{0.60mm}
\pgfcircle[stroke]{\pgfxy(190.00,60.00)}{0.60mm}
\pgfcircle[fill]{\pgfxy(200.00,60.00)}{0.60mm}
\pgfcircle[stroke]{\pgfxy(200.00,60.00)}{0.60mm}
\pgfcircle[fill]{\pgfxy(200.00,70.00)}{0.60mm}
\pgfcircle[stroke]{\pgfxy(200.00,70.00)}{0.60mm}
\pgfmoveto{\pgfxy(130.00,0.00)}\pgflineto{\pgfxy(130.00,0.00)}\pgfstroke
\pgfmoveto{\pgfxy(130.00,0.00)}\pgflineto{\pgfxy(220.00,0.00)}\pgfstroke
\pgfmoveto{\pgfxy(220.00,0.00)}\pgflineto{\pgfxy(217.20,0.70)}\pgflineto{\pgfxy(220.00,0.00)}\pgflineto{\pgfxy(217.20,-0.70)}\pgflineto{\pgfxy(220.00,0.00)}\pgfclosepath\pgffill
\pgfmoveto{\pgfxy(220.00,0.00)}\pgflineto{\pgfxy(217.20,0.70)}\pgflineto{\pgfxy(220.00,0.00)}\pgflineto{\pgfxy(217.20,-0.70)}\pgflineto{\pgfxy(220.00,0.00)}\pgfclosepath\pgfstroke
\pgfmoveto{\pgfxy(140.00,-10.00)}\pgflineto{\pgfxy(140.00,80.00)}\pgfstroke
\pgfmoveto{\pgfxy(140.00,80.00)}\pgflineto{\pgfxy(139.30,77.20)}\pgflineto{\pgfxy(140.00,80.00)}\pgflineto{\pgfxy(140.70,77.20)}\pgflineto{\pgfxy(140.00,80.00)}\pgfclosepath\pgffill
\pgfmoveto{\pgfxy(140.00,80.00)}\pgflineto{\pgfxy(139.30,77.20)}\pgflineto{\pgfxy(140.00,80.00)}\pgflineto{\pgfxy(140.70,77.20)}\pgflineto{\pgfxy(140.00,80.00)}\pgfclosepath\pgfstroke
\pgfsetlinewidth{0.60mm}\pgfmoveto{\pgfxy(90.00,40.00)}\pgflineto{\pgfxy(120.00,40.00)}\pgfstroke
\pgfmoveto{\pgfxy(120.00,40.00)}\pgflineto{\pgfxy(117.20,40.70)}\pgflineto{\pgfxy(120.00,40.00)}\pgflineto{\pgfxy(117.20,39.30)}\pgflineto{\pgfxy(120.00,40.00)}\pgfclosepath\pgffill
\pgfmoveto{\pgfxy(120.00,40.00)}\pgflineto{\pgfxy(117.20,40.70)}\pgflineto{\pgfxy(120.00,40.00)}\pgflineto{\pgfxy(117.20,39.30)}\pgflineto{\pgfxy(120.00,40.00)}\pgfclosepath\pgfstroke
\pgfmoveto{\pgfxy(120.00,35.00)}\pgflineto{\pgfxy(90.00,35.00)}\pgfstroke
\pgfmoveto{\pgfxy(90.00,35.00)}\pgflineto{\pgfxy(92.80,34.30)}\pgflineto{\pgfxy(90.00,35.00)}\pgflineto{\pgfxy(92.80,35.70)}\pgflineto{\pgfxy(90.00,35.00)}\pgfclosepath\pgffill
\pgfmoveto{\pgfxy(90.00,35.00)}\pgflineto{\pgfxy(92.80,34.30)}\pgflineto{\pgfxy(90.00,35.00)}\pgflineto{\pgfxy(92.80,35.70)}\pgflineto{\pgfxy(90.00,35.00)}\pgfclosepath\pgfstroke
\pgfputat{\pgfxy(105.00,45.00)}{\pgfbox[bottom,left]{\fontsize{6.83}{8.19}\selectfont \makebox[0pt]{$\ptd$}}}
\pgfputat{\pgfxy(105.00,28.00)}{\pgfbox[bottom,left]{\fontsize{6.83}{8.19}\selectfont \makebox[0pt]{$\dtp$}}}
\pgfsetdash{{0.30mm}{0.50mm}}{0mm}\pgfsetlinewidth{0.30mm}\pgfmoveto{\pgfxy(10.00,10.00)}\pgflineto{\pgfxy(20.00,20.00)}\pgfstroke
\pgfmoveto{\pgfxy(30.00,40.00)}\pgflineto{\pgfxy(40.00,50.00)}\pgfstroke
\pgfmoveto{\pgfxy(40.00,50.00)}\pgflineto{\pgfxy(50.00,60.00)}\pgfstroke
\pgfmoveto{\pgfxy(190.00,60.00)}\pgflineto{\pgfxy(180.00,60.00)}\pgflineto{\pgfxy(180.00,50.00)}\pgfstroke
\pgfmoveto{\pgfxy(180.00,50.00)}\pgflineto{\pgfxy(170.00,50.00)}\pgflineto{\pgfxy(170.00,40.00)}\pgfstroke
\pgfmoveto{\pgfxy(160.00,20.00)}\pgflineto{\pgfxy(150.00,20.00)}\pgflineto{\pgfxy(150.00,10.00)}\pgfstroke
\end{pgfpicture}%
$$
\caption{Example under mappings $\ptd$ and $\dtp$}
\label{fig:map}
\end{figure}

For example, as shown in Figure~\ref{fig:map},
we have
\begin{align*}
\ptd(ENNEENNNENEEN)&=ENDENNDDEN,\\
\dtp(ENDENNDDEN)&=ENNEENNNENEEN.
\end{align*}

\begin{thm}
Mappings $\ptd$ and $\dtp$ are bijections between $\P\nm(NE)$ and $\P\nm(D)$.
$$
\begin{tikzcd}
\P\nm&
\P\nm\\
\P\nm(D) \arrow[u, no head, "\cup"] \arrow[r, shift left=.7, "\ptd"] &
\P\nm(NE) \arrow[u, no head, "\cup"] \arrow[l, shift left=.7, "\dtp"] \\
\end{tikzcd}
$$
\end{thm}

\begin{proof}
$\dtp(\ptd(P))=P$ holds for $P \in \P\nm(D)$ and
$\ptd(\dtp(Q))=Q$ for $Q \in \P\nm(NE)$.
\end{proof}

From the previous section, the two sets $\di{\aug{\P\nm}(D, EENN)}$ and $\P\nm(NE, EN)$ are known to be equinumerous, that is,
\begin{align}
a(n,m)=h(n,m).
\label{eq:as}
\end{align}
%where $A = \min{floor((m+n)/2), floor(n/2)}$ and $B = min{floor((m+n-1)/2), floor((n-1)/2)}$.
Here, we provide a bijective proof of \eqref{eq:as} under $\ptd$ and $\dtp$.

\begin{thm}%[Main Result]
\label{thm:main}
Mappings $\ptd$ and $\dtp$  
between $\di{\aug{\P\nm}(D, EENN)}$ and $\P\nm(NE, EN)$
are bijections.
$$
\begin{tikzcd}
%\P\nm&
%\P\nm\\
\P\nm(D) \arrow[r, shift left=.7, "\ptd"] & %\arrow[u, no head, "\cup"] &
\P\nm(NE) \arrow[l, shift left=.7, "\dtp"] \\ %\arrow[u, no head, "\cup"] \\
\di{\aug{\P\nm}(D, EENN)} \arrow[u, no head, "\cup"] \arrow[r, shift left=.7, "\ptd"] &
\P\nm(NE, EN) \arrow[u, no head, "\cup"] \arrow[l, shift left=.7, "\dtp" ]
\end{tikzcd}
$$
\end{thm}

\begin{proof}%[Proof of Theorem~\ref{thm:main}]
It is sufficient to show that
\begin{align*}
P \in \di{\aug{\P\nm}(D, EENN)} &\Longrightarrow \ptd(P) \in \P\nm(NE, EN) \\
Q \in \P\nm(NE, EN) &\Longrightarrow \dtp(Q) \in \di{\aug{\P\nm}(D, EENN)}.
\end{align*}

\begin{enumerate}[(a)]
\item
For a given $P \in \di{\aug{\P\nm}(D, EENN)}$,
we have $P \in \P\nm(D)$ and $\ptd(P) \in \P\nm(NE)$.
Thus, $\ptd(P)$ avoids peaks $NE$.

%Under the peak-to-diagonal mapping $\ptd$,
%for a given North-East lattice path $P$ whose augmented lattice path $\aug{P}$ has no deep valleys,
%there exists no peaks and valleys in $\ptd(P)$.
%
%By the mapping $\ptd$, each $NE$ in $P$ is change into $D$, so there are no $NE$ in $\ptd(P)$.

Because $P \in \di{\aug{\P\nm}(D, EENN)}$ and $\aug{P}$ do not include a deep valley $EENN$,
only two possibilities exist for each valley $EN$ in $P$:
\begin{enumerate}[(i)]
\item If a valley $EN$ follows $N$ in $\aug{P}$, then $NEN$ should be changed to $DN$.
\item If a valley $EN$ preceds $E$ in $\aug{P}$, then $ENE$ should be changed to $ED$.
\end{enumerate}
Hence, $\ptd(P)$ avoids peaks $EN$, and we have $\ptd(P) \in \P\nm(NE, EN)$.

\item
For a given $Q \in \P\nm(NE, EN)$,
we have $Q \in \P\nm(NE)$ and $\dtp(Q) \in \P\nm(D)$.
Thus, $\aug{\dtp(Q)}$ avoids diagonal steps $D$.

%Under the diagonal-to-peak mapping $\dtp$,
%for a given Delannoy path $Q$ that has no peaks and valleys,
%there exists no deep valley in $\aug{\dtp(Q)}$.

Suppose $\aug{\dtp(Q)}$ includes a deep valley $EENN$.
Then, the path $\di{\ptd(\aug{\dtp(Q)})}$ includes a valley $EN$.
Because $$\di{\ptd(\aug{\dtp(Q)})}=Q,$$
$Q$ includes a valley $EN$,
which contradicts $Q \in \P\nm(NE, EN)$.
Hence, $\aug{\dtp(Q)}$ avoids deep valleys $EENN$, and we have $\dtp(Q) \in \di{\aug{\P\nm}(D, EENN)}$.
\end{enumerate}
\end{proof}

We complete a combinatorial proof of \eqref{eq:s_gen}
because \eqref{eq:b_gen} is proved by the principle of inclusion and exclusion,
\eqref{eq:tau} is proved bijectively by $\tau$,
and \eqref{eq:as} is proved bijectively by $\ptd$ and $\dtp$.

\section{$k$-Schr\"oder paths}
\label{sec:path}

Let $k$ be a positive integer.
%A subset of $\Z^2$ is called a \emph{region}.
Define a \emph{region} $\Z^2(k)$ by
$$\Z^2(k)=\set{(x,y) \in \Z^2 : y \ge k x}.$$

Let $\P\nm\k$ be the set of Delannoy paths $P \in \P\nm$, whose vertices lie on the region $\Z^2(k)$.
Song \cite{Son05} introduced a \emph{$k$-Schr\"oder path} $P$ of size $n$
if $P \in \P\nkn\k$ for nonnegative integers $n$,
where the cardinality of $\P\nkn\k$ equals 
$$
\#\P\nkn\k
= \sum_{d=0}^n \frac{1}{kn-d+1}\binom{kn+n-d}{kn-d,n-d,d}
= \frac{1}{n}\sum_{j=1}^{n} \binom{kn}{j-1} \binom{n}{j} 2^j.
$$
Recently, Yang and Jiang \cite{YJ21} showed that the cardinality of $\P\nm\k$ equals 
$$
\#\P_{n,m}\k
= \sum_{d=0}^n \frac{m-kn+1}{m-d+1}\binom{m+n-d}{m-d,n-d,d}
= \frac{m-kn+1}{n}\sum_{j=1}^{n} \binom{m}{j-1} \binom{n}{j} 2^{j},
$$
with $\#\P_{n,m}\k = 0 $ if $m < kn$.

%A region $R$ is called a \emph{triangle-convex}
%if $(x,y)\in R$ and $(x+1, y+1)\in R$ then $(x, y+1) \in R$ for any lattice point $(x,y)$.

%예를 들어 $\Z^2$는 \emph{triangle-convex}이다. 또한 임의의 양의 실수 $k$에 대하여
%은 \emph{triangle-convex}이다.

%\seo{$\P\nm\k$의 정의가 필요. 또한 $k$-Schroder[Song05, YJ21] 와의 연계성 언급}

%\begin{center}
%\begin{tikzcd}
%\P\nm&
%\P\nm\\
%\P\nm(D) \arrow[u, no head, "\cup"] \arrow[r, shift left=.7, "\ptd"] &
%\P\nm(NE) \arrow[u, no head, "\cup"] \arrow[l, shift left=.7, "\dtp"] \\
%\di{\aug{\P\nm}(D, EENN)} \arrow[u, no head, "\cup"] \arrow[r, shift left=.7, "\ptd"] &
%\P\nm(NE, EN) \arrow[u, no head, "\cup"] \arrow[l, shift left=.7, "\dtp" ] \\
%\di{\aug{\P\k\nm}(D, EENN)} \arrow[u, no head, "\cup"] & %\arrow[r, shift left=.7, "\ptd"] &
%\P\k\nm(NE,EN) \arrow[u, no head, "\cup"] %\arrow[l, shift left=.7, "\dtp" ]
%\end{tikzcd}
%\end{center}

Let $\di{\aug{\P\nm\k}(D, EENN)}$ be the set of North-East lattice paths $P$ in $\P\nm\k$,
where the augmented lattice path $\aug{P}$ does not have a deep valley $EENN$,
and let $\P\nm\k(NE, EN)$ be the set of Delannoy paths without peaks and valleys in $\P\nm\k$.

\begin{cor}
\label{cor:main}
Mappings $\ptd$ from $\di{\aug{\P\k\nm}(D, EENN)}$ to $\P\k\nm(NE, EN)$
and
$\dtp$ from \linebreak $\P\k\nm(NE, EN)$ to $\di{\aug{\P\k\nm}(D, EENN)}$
are bijections.
$$
\begin{tikzcd}
%\P\nm&
%\P\nm\\
%\P\nm(D) \arrow[r, shift left=.7, "\ptd"] \arrow[u, no head, "\cup"] &
%\P\nm(NE) \arrow[l, shift left=.7, "\dtp"] \arrow[u, no head, "\cup"] \\
\di{\aug{\P\nm}(D, EENN)} \arrow[r, shift left=.7, "\ptd"] & % \arrow[u, no head, "\cup"] &
\P\nm(NE, EN) \arrow[l, shift left=.7, "\dtp" ] \\ %\arrow[u, no head, "\cup"] \\
\di{\aug{\P\k\nm}(D, EENN)} \arrow[u, no head, "\cup"] \arrow[r, shift left=.7, "\ptd"] &
\P\k\nm(NE,EN) \arrow[u, no head, "\cup"] \arrow[l, shift left=.7, "\dtp" ]
\end{tikzcd}
$$
\end{cor}

\begin{proof}
Because, for any lattice point $(x,y)$,
if $(x,y)\in \Z^2(k)$ and $(x+1, y+1)\in \Z^2(k)$ then $(x, y+1) \in \Z^2(k)$,
we know that
\begin{align*}
P \in \di{\aug{\P\nm\k}(D, EENN)} &\Longrightarrow \ptd(P) \in \P\nm\k(NE, EN) \\
Q \in \P\nm\k(NE, EN) &\Longrightarrow \dtp(Q) \in \di{\aug{\P\nm\k}(D, EENN)},
\end{align*}
which completes the proof.
\end{proof}

%\begin{thm}
%\begin{enumerate}[(i)]
%\item Given two lattice points $A$ and $B$,
%there exists a bijection between
%the set of lattice paths from $A$ to $B$ using steps $E=(0,1)$ and $N=(1,0)$ without $EENN$ and
%the set of lattice paths from $A$ to $B$ using steps $E=(0,1)$, $N=(1,0)$ and $D=(1,1)$ without $NE$ and $EN$.
%
%\item There exists a bijection between
%the set of Dyck paths from $(0,0)$ to $(2n,0)$ using steps $U=(1,1)$ and $D=(1,-1)$ without $DDUU$ and
%the set of Schr\"{o}der paths from $(0,0)$ to $(2n,0)$ using steps $U=(1,1)$, $D=(1,-1)$ and $H=(2,0)$ without $UD$ and $DU$.
%\end{enumerate}
%\end{thm}
%
%\begin{proof}
%\begin{enumerate}[(i)]
%\item The bijection $\ptd$ convert $NE$ to $D$.
%\item The bijection $\ptd$ convert $UD$ to $H$.
%\end{enumerate}
%\end{proof}

%\begin{rmk}[따름정리~\ref{cor:main}과 동치]
%There exists a bijection between
%the set of Dyck paths from $(0,0)$ to $(2n,0)$ using up steps $\U=(1,1)$ and down steps $\D=(1,-1)$ without $\D\D\U\U$ and
%the set of Schr\"{o}der paths from $(0,0)$ to $(2n,0)$ using up steps $\U=(1,1)$, down steps $\D=(1,-1)$ and horizontal steps $\H=(2,0)$ without $\U\D$ and $\D\U$.
%\end{rmk}

\section{$k$-Schr\"oder paths without peaks and valleys}
\label{sec:without}

Because every $k$-Schr\"oder path should begin with $N$ and end with $E$,
for a $k$-Schr\"oder path $P$ without deep valleys, $\aug{P}$ does not have deep valleys.
Thus, we have
%$$\di{\aug{\P\k\nkn}(D,EENN)} = \P\k\nkn(D,EENN).$$
$$\P\k\nkn(D,EENN) = \di{\aug{\P\k\nkn}(D,EENN)}.$$
%By the bijection $\ptd$ or $\dtp$,
From Corollary~\ref{cor:main}, we obtain
$$\# \P\k\nkn(D,EENN) = \# \P\k\nkn(NE,EN).$$

%Let us decompose the set $\P\k\nkn(NE,EN)$.
%By reading the sequnece of steps reversly, there exists a bijection between two sets
%$\P\k\nkn(NE,EN)$ and $\P^{(1/k)}_{kn,n}(NE,EN)$. \seo{그림에 따라 문장이 바뀔 수 있음}
%So we can analyze the set $\P^{(1/k)}_{kn,n}(NE,EN)$.
%Define $D\P^{(1/k)}_{kn,n}(NE,EN)$ the set of $P \in \P^{(1/k)}_{kn,n}(NE,EN)$ begining with diagonal step $D$.
%Also, define $N\P^{(1/k)}_{kn,n}(NE,EN)$ the set of $P \in \P^{(1/k)}_{kn,n}(NE,EN)$ begining with north step $N$.
%Note that
%$$\P^{(1/k)}_{kn,n}(NE,EN) = \{\emptyset\} \cup D\P^{(1/k)}_{kn,n}(NE,EN) \cup N\P^{(1/k)}_{kn,n}(NE,EN). $$

%\subsection*{$k$-Schr\"oder paths without peaks and valleys}

Let us divide $\P\k\nkn(NE,EN)$ by the last step.
Define $D\P\k\nkn(NE,EN)$ as the set of $P \in \P\k\nkn(NE,EN)$ ending with diagonal step $D$.
In addition, define $E\P\k\nkn(NE,EN)$ as the set of $P \in \P\k\nkn(NE,EN)$ ending with east step $E$.
By definition, it holds that
%$$\P\k\nkn(NE,EN) = \set{\emptyset} \cup D\P\k\nkn(NE,EN) \cup E\P\k\nkn(NE,EN).$$
$$\# \P\k\nkn(NE,EN) = \delta_{n,0} + \# D\P\k\nkn(NE,EN) + \# E\P\k\nkn(NE,EN).$$
Define three generating functions $F$, $F_D$, and $F_E$ as
\begin{align*}
F &=F\k(x) = \sum_{n\ge0} \#{\P\k\nkn(NE,EN)} x^n, \\
F_D &=F_D\k(x) = \sum_{n\ge1} \# D\P\k\nkn(NE,EN) x^n, \\
F_E &=F_E\k(x) = \sum_{n\ge1} \# E\P\k\nkn(NE,EN) x^n.
\end{align*}
These satisfy the identity
\begin{align*}
F &= 1 + F_D + F_E. %\label{eq:decom1}
\end{align*}

%Since $D\P\k\nkn(NE,EN)$ corresponds the set $\P\k_{n-1,kn-1}(NE,EN)$
%and $E\P\k\nkn(NE,EN)$ corresponds the set $\bigcup_{j=0}^{n-2}\P\k_{j,kn-1}(NE,EN)$,
%we have
%$$\# \P^{k}\nkn(NE,EN) = 1 + \# \P\k_{n-1,kn-1}(NE,EN) + \sum_{j=0}^{n-2}\# \P\k_{j,kn-1}(NE,EN).$$
%Define three generating functions $F$, $F_D$, and $F_E$ of
%\begin{align*}
%F &=F\k(x) = \sum_{n\ge0} \#{\P\k\nkn(NE,EN)} x^n \\
%F_D &=F_D\k(x) = \sum_{n\ge1} \#{\P\k_{n-1,kn-1}(NE,EN)} x^n \\
%F_E &=F_E\k(x) = \sum_{n\ge1} \left(\sum_{j=0}^{n-2}\# \P\k_{j,kn-1}(NE,EN)\right) x^n.
%\end{align*}
%Here $F_D$ (resp. $F_E$) is the generating function for the number of paths $P$ in $\P\k\nkn(NE,EN)$ ending with $D$ (resp. $E$).

\begin{figure}
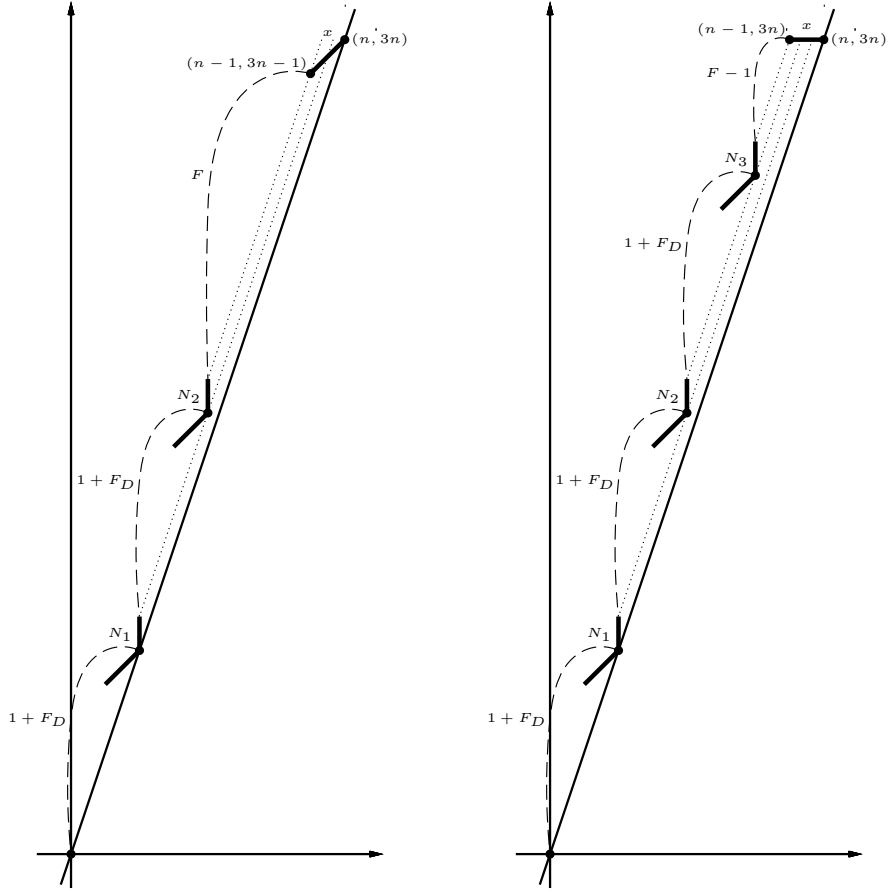

%$$\input{decomposition-k.TpX}$$
$$
\centering
\begin{pgfpicture}{-6.67mm}{-6.50mm}{106.25mm}{114.50mm}
\pgfsetxvec{\pgfpoint{0.45mm}{0mm}}
\pgfsetyvec{\pgfpoint{0mm}{0.45mm}}
\color[rgb]{0,0,0}\pgfsetlinewidth{0.30mm}\pgfsetdash{}{0mm}
\pgfmoveto{\pgfxy(80.00,250.00)}\pgflineto{\pgfxy(80.00,250.00)}\pgfstroke
\pgfmoveto{\pgfxy(-10.00,0.00)}\pgflineto{\pgfxy(90.00,0.00)}\pgfstroke
\pgfmoveto{\pgfxy(90.00,0.00)}\pgflineto{\pgfxy(87.20,0.70)}\pgflineto{\pgfxy(90.00,0.00)}\pgflineto{\pgfxy(87.20,-0.70)}\pgflineto{\pgfxy(90.00,0.00)}\pgfclosepath\pgffill
\pgfmoveto{\pgfxy(90.00,0.00)}\pgflineto{\pgfxy(87.20,0.70)}\pgflineto{\pgfxy(90.00,0.00)}\pgflineto{\pgfxy(87.20,-0.70)}\pgflineto{\pgfxy(90.00,0.00)}\pgfclosepath\pgfstroke
\pgfmoveto{\pgfxy(0.00,-10.00)}\pgflineto{\pgfxy(0.00,250.00)}\pgfstroke
\pgfmoveto{\pgfxy(0.00,250.00)}\pgflineto{\pgfxy(-0.70,247.20)}\pgflineto{\pgfxy(0.00,250.00)}\pgflineto{\pgfxy(0.70,247.20)}\pgflineto{\pgfxy(0.00,250.00)}\pgfclosepath\pgffill
\pgfmoveto{\pgfxy(0.00,250.00)}\pgflineto{\pgfxy(-0.70,247.20)}\pgflineto{\pgfxy(0.00,250.00)}\pgflineto{\pgfxy(0.70,247.20)}\pgflineto{\pgfxy(0.00,250.00)}\pgfclosepath\pgfstroke
\pgfmoveto{\pgfxy(89.00,243.00)}\pgflineto{\pgfxy(89.00,243.00)}\pgfstroke
\pgfmoveto{\pgfxy(89.00,243.00)}\pgflineto{\pgfxy(89.00,243.00)}\pgfstroke
\pgfmoveto{\pgfxy(83.00,249.00)}\pgflineto{\pgfxy(-3.00,-9.00)}\pgfstroke
\pgfsetdash{{0.15mm}{0.50mm}}{0mm}\pgfsetlinewidth{0.15mm}\pgfmoveto{\pgfxy(76.61,239.98)}\pgflineto{\pgfxy(20.00,70.00)}\pgfstroke
\pgfmoveto{\pgfxy(73.25,239.98)}\pgflineto{\pgfxy(40.00,140.00)}\pgfstroke
\pgfsetdash{}{0mm}\pgfsetlinewidth{0.60mm}\pgfmoveto{\pgfxy(80.00,240.00)}\pgflineto{\pgfxy(70.00,230.00)}\pgfstroke
\pgfmoveto{\pgfxy(40.00,140.00)}\pgflineto{\pgfxy(40.00,130.00)}\pgfstroke
\pgfmoveto{\pgfxy(20.00,70.00)}\pgflineto{\pgfxy(20.00,60.00)}\pgfstroke
\pgfmoveto{\pgfxy(40.00,130.00)}\pgflineto{\pgfxy(30.00,120.00)}\pgfstroke
\pgfmoveto{\pgfxy(20.00,60.00)}\pgflineto{\pgfxy(10.00,50.00)}\pgfstroke
\pgfsetdash{{2.00mm}{1.00mm}}{0mm}\pgfsetlinewidth{0.15mm}\pgfmoveto{\pgfxy(40.00,130.00)}\pgfcurveto{\pgfxy(36.83,131.48)}{\pgfxy(33.17,131.48)}{\pgfxy(30.00,130.00)}\pgfcurveto{\pgfxy(22.81,126.64)}{\pgfxy(20.81,118.06)}{\pgfxy(20.00,110.00)}\pgfcurveto{\pgfxy(18.66,96.70)}{\pgfxy(18.66,83.30)}{\pgfxy(20.00,70.00)}\pgfstroke
\pgfmoveto{\pgfxy(20.00,60.00)}\pgfcurveto{\pgfxy(16.83,61.48)}{\pgfxy(13.17,61.48)}{\pgfxy(10.00,60.00)}\pgfcurveto{\pgfxy(2.81,56.64)}{\pgfxy(0.81,48.06)}{\pgfxy(0.00,40.00)}\pgfcurveto{\pgfxy(-1.34,26.70)}{\pgfxy(-1.34,13.30)}{\pgfxy(0.00,0.00)}\pgfstroke
\pgfmoveto{\pgfxy(70.00,230.00)}\pgfcurveto{\pgfxy(66.73,230.92)}{\pgfxy(63.27,230.92)}{\pgfxy(60.00,230.00)}\pgfcurveto{\pgfxy(44.45,225.60)}{\pgfxy(40.57,207.10)}{\pgfxy(40.00,190.00)}\pgfcurveto{\pgfxy(39.45,173.34)}{\pgfxy(39.45,156.66)}{\pgfxy(40.00,140.00)}\pgfstroke
\pgfputat{\pgfxy(39.00,199.00)}{\pgfbox[bottom,left]{\rotatebox{0.00}{\fontsize{5.12}{6.15}\selectfont \smash{\makebox[0pt][r]{$F$}}}}}
\pgfputat{\pgfxy(19.00,109.00)}{\pgfbox[bottom,left]{\rotatebox{0.00}{\fontsize{5.12}{6.15}\selectfont \smash{\makebox[0pt][r]{$1+F_D$}}}}}
\pgfputat{\pgfxy(-1.00,39.00)}{\pgfbox[bottom,left]{\rotatebox{0.00}{\fontsize{5.12}{6.15}\selectfont \smash{\makebox[0pt][r]{$1+F_D$}}}}}
\pgfputat{\pgfxy(75.00,241.00)}{\pgfbox[bottom,left]{\rotatebox{0.00}{\fontsize{5.12}{6.15}\selectfont \smash{\makebox[0pt]{$x$}}}}}
\pgfsetdash{}{0mm}\pgfsetlinewidth{0.30mm}\pgfmoveto{\pgfxy(209.00,243.00)}\pgflineto{\pgfxy(209.00,243.00)}\pgfstroke
\pgfmoveto{\pgfxy(209.00,243.00)}\pgflineto{\pgfxy(209.00,243.00)}\pgfstroke
\pgfmoveto{\pgfxy(220.00,250.00)}\pgflineto{\pgfxy(220.00,250.00)}\pgfstroke
\pgfmoveto{\pgfxy(130.00,0.00)}\pgflineto{\pgfxy(230.00,0.00)}\pgfstroke
\pgfmoveto{\pgfxy(230.00,0.00)}\pgflineto{\pgfxy(227.20,0.70)}\pgflineto{\pgfxy(230.00,0.00)}\pgflineto{\pgfxy(227.20,-0.70)}\pgflineto{\pgfxy(230.00,0.00)}\pgfclosepath\pgffill
\pgfmoveto{\pgfxy(230.00,0.00)}\pgflineto{\pgfxy(227.20,0.70)}\pgflineto{\pgfxy(230.00,0.00)}\pgflineto{\pgfxy(227.20,-0.70)}\pgflineto{\pgfxy(230.00,0.00)}\pgfclosepath\pgfstroke
\pgfmoveto{\pgfxy(140.00,-10.00)}\pgflineto{\pgfxy(140.00,250.00)}\pgfstroke
\pgfmoveto{\pgfxy(140.00,250.00)}\pgflineto{\pgfxy(139.30,247.20)}\pgflineto{\pgfxy(140.00,250.00)}\pgflineto{\pgfxy(140.70,247.20)}\pgflineto{\pgfxy(140.00,250.00)}\pgfclosepath\pgffill
\pgfmoveto{\pgfxy(140.00,250.00)}\pgflineto{\pgfxy(139.30,247.20)}\pgflineto{\pgfxy(140.00,250.00)}\pgflineto{\pgfxy(140.70,247.20)}\pgflineto{\pgfxy(140.00,250.00)}\pgfclosepath\pgfstroke
\pgfmoveto{\pgfxy(229.00,243.00)}\pgflineto{\pgfxy(229.00,243.00)}\pgfstroke
\pgfmoveto{\pgfxy(229.00,243.00)}\pgflineto{\pgfxy(229.00,243.00)}\pgfstroke
\pgfmoveto{\pgfxy(223.00,249.00)}\pgflineto{\pgfxy(137.00,-9.00)}\pgfstroke
\pgfsetdash{{0.15mm}{0.50mm}}{0mm}\pgfsetlinewidth{0.15mm}\pgfmoveto{\pgfxy(216.66,239.98)}\pgflineto{\pgfxy(160.00,70.00)}\pgfstroke
\pgfmoveto{\pgfxy(213.31,239.98)}\pgflineto{\pgfxy(180.00,140.00)}\pgfstroke
\pgfsetdash{}{0mm}\pgfsetlinewidth{0.60mm}\pgfmoveto{\pgfxy(220.00,240.00)}\pgflineto{\pgfxy(210.00,240.00)}\pgfstroke
\pgfmoveto{\pgfxy(180.00,140.00)}\pgflineto{\pgfxy(180.00,130.00)}\pgfstroke
\pgfmoveto{\pgfxy(160.00,70.00)}\pgflineto{\pgfxy(160.00,60.00)}\pgfstroke
\pgfmoveto{\pgfxy(180.00,130.00)}\pgflineto{\pgfxy(170.00,120.00)}\pgfstroke
\pgfmoveto{\pgfxy(160.00,60.00)}\pgflineto{\pgfxy(150.00,50.00)}\pgfstroke
\pgfsetdash{{2.00mm}{1.00mm}}{0mm}\pgfsetlinewidth{0.15mm}\pgfmoveto{\pgfxy(180.00,130.00)}\pgfcurveto{\pgfxy(176.83,131.48)}{\pgfxy(173.17,131.48)}{\pgfxy(170.00,130.00)}\pgfcurveto{\pgfxy(162.81,126.64)}{\pgfxy(160.81,118.06)}{\pgfxy(160.00,110.00)}\pgfcurveto{\pgfxy(158.66,96.70)}{\pgfxy(158.66,83.30)}{\pgfxy(160.00,70.00)}\pgfstroke
\pgfmoveto{\pgfxy(160.00,60.00)}\pgfcurveto{\pgfxy(156.83,61.48)}{\pgfxy(153.17,61.48)}{\pgfxy(150.00,60.00)}\pgfcurveto{\pgfxy(142.81,56.64)}{\pgfxy(140.81,48.06)}{\pgfxy(140.00,40.00)}\pgfcurveto{\pgfxy(138.66,26.70)}{\pgfxy(138.66,13.30)}{\pgfxy(140.00,0.00)}\pgfstroke
\pgfmoveto{\pgfxy(210.00,240.00)}\pgfcurveto{\pgfxy(208.42,240.74)}{\pgfxy(206.58,240.74)}{\pgfxy(205.00,240.00)}\pgfcurveto{\pgfxy(201.41,238.32)}{\pgfxy(200.41,234.03)}{\pgfxy(200.00,230.00)}\pgfcurveto{\pgfxy(199.33,223.35)}{\pgfxy(199.33,216.65)}{\pgfxy(200.00,210.00)}\pgfstroke
\pgfputat{\pgfxy(199.00,229.00)}{\pgfbox[bottom,left]{\rotatebox{0.00}{\fontsize{5.12}{6.15}\selectfont \smash{\makebox[0pt][r]{$F-1$}}}}}
\pgfputat{\pgfxy(159.00,109.00)}{\pgfbox[bottom,left]{\rotatebox{0.00}{\fontsize{5.12}{6.15}\selectfont \smash{\makebox[0pt][r]{$1+F_D$}}}}}
\pgfputat{\pgfxy(139.00,39.00)}{\pgfbox[bottom,left]{\rotatebox{0.00}{\fontsize{5.12}{6.15}\selectfont \smash{\makebox[0pt][r]{$1+F_D$}}}}}
\pgfputat{\pgfxy(215.00,243.00)}{\pgfbox[bottom,left]{\rotatebox{0.00}{\fontsize{5.12}{6.15}\selectfont \smash{\makebox[0pt]{$x$}}}}}
\pgfsetdash{{0.15mm}{0.50mm}}{0mm}\pgfmoveto{\pgfxy(210.00,240.00)}\pgflineto{\pgfxy(200.00,210.00)}\pgfstroke
\pgfsetdash{{2.00mm}{1.00mm}}{0mm}\pgfmoveto{\pgfxy(200.00,200.00)}\pgfcurveto{\pgfxy(196.83,201.48)}{\pgfxy(193.17,201.48)}{\pgfxy(190.00,200.00)}\pgfcurveto{\pgfxy(182.81,196.64)}{\pgfxy(180.81,188.06)}{\pgfxy(180.00,180.00)}\pgfcurveto{\pgfxy(178.66,166.70)}{\pgfxy(178.66,153.30)}{\pgfxy(180.00,140.00)}\pgfstroke
\pgfsetdash{}{0mm}\pgfsetlinewidth{0.60mm}\pgfmoveto{\pgfxy(200.00,200.00)}\pgflineto{\pgfxy(190.00,190.00)}\pgfstroke
\pgfmoveto{\pgfxy(200.00,210.00)}\pgflineto{\pgfxy(200.00,200.00)}\pgfstroke
\pgfputat{\pgfxy(179.00,179.00)}{\pgfbox[bottom,left]{\rotatebox{0.00}{\fontsize{5.12}{6.15}\selectfont \smash{\makebox[0pt][r]{$1+F_D$}}}}}
\pgfputat{\pgfxy(82.00,239.00)}{\pgfbox[bottom,left]{\fontsize{5.12}{6.15}\selectfont $(n,3n)$}}
\pgfputat{\pgfxy(222.00,239.00)}{\pgfbox[bottom,left]{\fontsize{5.12}{6.15}\selectfont $(n,3n)$}}
\pgfputat{\pgfxy(69.00,232.00)}{\pgfbox[bottom,left]{\fontsize{5.12}{6.15}\selectfont \makebox[0pt][r]{$(n-1,3n-1)$}}}
\pgfputat{\pgfxy(209.00,242.00)}{\pgfbox[bottom,left]{\fontsize{5.12}{6.15}\selectfont \makebox[0pt][r]{$(n-1,3n)$}}}
\pgfcircle[fill]{\pgfxy(80.00,240.00)}{0.45mm}
\pgfsetlinewidth{0.30mm}\pgfcircle[stroke]{\pgfxy(80.00,240.00)}{0.45mm}
\pgfcircle[fill]{\pgfxy(70.00,230.00)}{0.45mm}
\pgfcircle[stroke]{\pgfxy(70.00,230.00)}{0.45mm}
\pgfcircle[fill]{\pgfxy(210.00,240.00)}{0.45mm}
\pgfcircle[stroke]{\pgfxy(210.00,240.00)}{0.45mm}
\pgfcircle[fill]{\pgfxy(220.00,240.00)}{0.45mm}
\pgfcircle[stroke]{\pgfxy(220.00,240.00)}{0.45mm}
\pgfcircle[fill]{\pgfxy(140.00,0.00)}{0.45mm}
\pgfcircle[stroke]{\pgfxy(140.00,0.00)}{0.45mm}
\pgfcircle[fill]{\pgfxy(0.00,0.00)}{0.45mm}
\pgfcircle[stroke]{\pgfxy(0.00,0.00)}{0.45mm}
\pgfcircle[fill]{\pgfxy(180.00,130.00)}{0.45mm}
\pgfcircle[stroke]{\pgfxy(180.00,130.00)}{0.45mm}
\pgfcircle[fill]{\pgfxy(160.00,60.00)}{0.45mm}
\pgfcircle[stroke]{\pgfxy(160.00,60.00)}{0.45mm}
\pgfcircle[fill]{\pgfxy(200.00,200.00)}{0.45mm}
\pgfcircle[stroke]{\pgfxy(200.00,200.00)}{0.45mm}
\pgfcircle[fill]{\pgfxy(40.00,130.00)}{0.45mm}
\pgfcircle[stroke]{\pgfxy(40.00,130.00)}{0.45mm}
\pgfcircle[fill]{\pgfxy(20.00,60.00)}{0.45mm}
\pgfcircle[stroke]{\pgfxy(20.00,60.00)}{0.45mm}
\pgfputat{\pgfxy(18.00,64.00)}{\pgfbox[bottom,left]{\fontsize{5.12}{6.15}\selectfont \makebox[0pt][r]{$N_1$}}}
\pgfputat{\pgfxy(158.00,64.00)}{\pgfbox[bottom,left]{\fontsize{5.12}{6.15}\selectfont \makebox[0pt][r]{$N_1$}}}
\pgfputat{\pgfxy(38.00,134.00)}{\pgfbox[bottom,left]{\fontsize{5.12}{6.15}\selectfont \makebox[0pt][r]{$N_2$}}}
\pgfputat{\pgfxy(178.00,134.00)}{\pgfbox[bottom,left]{\fontsize{5.12}{6.15}\selectfont \makebox[0pt][r]{$N_2$}}}
\pgfputat{\pgfxy(198.00,204.00)}{\pgfbox[bottom,left]{\fontsize{5.12}{6.15}\selectfont \makebox[0pt][r]{$N_3$}}}
\end{pgfpicture}%
$$
\caption{Decomposition of $3$-Schr\"{o}der paths}
\label{fig:decomp}
\end{figure}

On the left of Figure~\ref{fig:decomp},
every path $P \in D\P\k\nkn(NE,EN)$ can be decomposed into $k$ subpaths as follows.
For $j=1, \dots, k-1$, let $N_j$ be the last north steps between lines $y=kx+(j-1)$ and $y=kx+j$. 
Consider $k$ subpaths $P_1, \dots, P_k$ 
by removing $(k-1)$ north steps $N_1, \dots, N_{k-1}$ and the last diagonal step $D$.
Because $P$ avoids peaks and valleys, 
subpaths $P_1, \dots, P_{k-1}$ are unable to end with an east step $E$,
but $P_k$ does not have this restriction.
This decomposition yields the identity
\begin{align}
F_D &= (1+F_D)^{k-1} F x. \label{eq:decom2}
\end{align}

On the right in Figure~\ref{fig:decomp},
every path $P \in E\P\k\nkn(NE,EN)$ can be decomposed into $(k+1)$ parts as follows.
For $j=1, \dots, k$, let $N_j$ be the last north steps between lines $y=kx+(j-1)$ and $y=kx+j$.
Consider $k$ subpaths $P_1, \dots, P_{k+1}$ 
by removing $k$ north steps $N_1, \dots, N_k$ and the last diagonal step $E$.
Because $P$ avoids peaks and valleys, 
subpaths $P_1, \dots, P_k$ are unable to end with an east step $E$,
but $P_{k+1}$ should not be a path of length $0$.
This decomposition yields the identity
\begin{align}
F_E &= (1+F_D)^{k} (F-1) x. \label{eq:decom3}
\end{align}

%By above two decompositions, three generating functions $F$, $F_D$, and $F_E$ satisfy
%\begin{align}
%F &= 1 + F_D + F_E \label{eq:decom1}\\
%F_D &= (1+F_D)^{k-1} F x \label{eq:decom2}\\
%F_E &= (1+F_D)^{k} (F_D + F_E) x. \label{eq:decom3}
%\end{align}

%\shin{3월 8일 오늘은 여기까지..}

Mutiplying \eqref{eq:decom2} by $(1+F_D)$ and subtracting this by \eqref{eq:decom3}, $F_E$ is expressed by $F_D$ as follows:
\begin{align}
F_E = (1+F_D)F_D - x(1+F_D)^k. \label{eq:F_E}
\end{align}
%Substituting $F_E$ in \eqref{eq:decom3} by \eqref{eq:F_E}, we have
%\begin{align*}
%(1+F_D)F_D - x(1+F_D)^k &= x (F_D + (1+F_D)F_D - x(1+F_D)^k) (1+F_D)^{k}.
%\end{align*}
%and, simplifying it,
%\begin{align}
%F_D &= x(1+F_D)^{k+1} - x^2(1+F_D)^{2k-1}. \label{eq:F_D}
%\end{align}
Substituting $F_E$ in \eqref{eq:decom3} by \eqref{eq:F_E} and simplifying, we have
\begin{align}
F_D &= x(1+F_D)^{k+1} - x^2(1+F_D)^{2k-1}. \label{eq:F_D}
\end{align}
Because the right-hand sides of \eqref{eq:decom2} and \eqref{eq:F_D} are the same, $F$ is expressed by $F_D$ as follows:
\begin{align}
%x F (1+F_D)^{k-1} &= x(1+F_D)^{k+1} - x^2(1+F_D)^{2k-1} \\
F &= (1+F_D)^2 - x(1+F_D)^k.
\end{align}
%Fi$F$, $F_D$와 \eqref{eq:decom1}에 의하여 $F_E$을 구할 수 있다.
To obtain all generating functions, 
calculating $F_D$ from \eqref{eq:F_D} is sufficient.

%\begin{align*}
%F &= 1 + x + 2x^2 + 5x^3 + 13x^4 + 35 x^5 + 97 x^6 + 275 x^7 + \cdots,\\
%F_D &= x + x^2 + 2x^3 + 5x^4 + 13x^5 + 35 x^6 + 97 x^7 + \cdots,\\
%F_E &= x^2 + 3x^3 + 8x^4 + 22 x^5 + 62 x^6 + 178x^7 +\cdots.
%\end{align*}

%%%%%%%%%%%%%%%%%%%%%%%%%%%%%%%%%%%%%%%%%%%%%%%%%%%%%%%%%%%%%%%%%%%%%%%%%%%%%%
\subsection*{Enumerations of $\P\one_{n,n}(D,EENN)$ and $\P\one_{n,n}(NE,EN)$}

For $k=1$, we have
\begin{align*}
F\one(x) &= \frac{(1-x)^2 - \sqrt{(1-x)^4-4x^2(1-x)}}{2x^2}\\
&= 1 + x + 2x^2 + 5x^3 + 13x^4 + 35 x^5 + 97 x^6 + 275 x^7 + \cdots,\\
F_D\one(x) %&= x F(x) \\
&= \frac{(1-x)^2 - \sqrt{(1-x)^4-4x^2(1-x)}}{2x}\\
&= x + x^2 + 2x^3 + 5x^4 + 13x^5 + 35 x^6 + 97 x^7 + \cdots,\\
F_E\one(x) %&= F(x)-F_D(x)-1\\
&= \frac{(1-x)^3 - (1-x)\sqrt{(1-x)^4-4x^2(1-x)}}{2x^2} - 1\\
&= x^2 + 3x^3 + 8x^4 + 22 x^5 + 62 x^6 + 178x^7 +\cdots,
\end{align*}
where the sequence of the coefficients in $F\one(x)$ and $F_E\one(x)$ are given by entries A086581 and A188464 in the OEIS \cite{Slo18}.

\begin{thm}% [k=1]
For $k=1$, the coefficients of $x^n$ of $F$, $F_D$, and $F_E$ are as follows:
\begin{align*}
[x^n] F\one(x)
&= \sum_{m\ge0} \frac{1}{m+1} {2m \choose m}{n+m \choose 3m}, \\
[x^n] F_D\one(x) 
%&= [x^{n-1}] F\one(x)
&= \sum_{m\ge0} \frac{1}{m+1} {2m \choose m}{n+m-1 \choose 3m}, \\
[x^{n}] F_E\one(x) 
%&= [x^n] F\one(x) - [x^n] F_D\one(x) 
&= \sum_{m\ge1} \frac{1}{m+1} {2m \choose m}{n+m-1 \choose 3m-1}.
\end{align*}
%\seo{A086581에 Ref 존재} %\cite{Eli21}
\end{thm}
\begin{proof}
Let
$$C(x) = \frac{1-\sqrt{1-4x}}{2x} = \sum_{n\ge 0} \frac{1}{n+1}\binom{2n}{n}.$$
It is well-known that $C(x)$ satisfies
$$C(x) = 1 + xC(x)^2.$$
Suppose $H(x)$ satisfies the equation
$$a(x) H(x) = 1+ b(x) H(x)^2, $$
where $a(0)\ne0$ and $b(0)= 0$.
Evidently, $H(x)$ is expressed with $C(x)$, $a(x)$, and $b(x)$ by
\begin{align} \label{eq:H}
H(x) = \frac{1}{a(x)} C\left(\frac{b(x)}{a(x)^2}\right).
\end{align}
From \eqref{eq:decom2} and \eqref{eq:F_D}, we obtain
\begin{align*}
F_D &= x(1+F_D)^{2} - x^2(1+F_D) \quad \text{ and } \quad 
F_D = xF
\end{align*}
and, eliminating $F_D$,
$$
(1-x)F=1+\frac{x^2}{1-x}F^2.
$$
According to \eqref{eq:H}, $F=F\one(x)$ is expressed by
%Since $a(x)=1-x$ and $b(x)= \frac{x^2}{1-x}$,
\begin{align*}
F\one(x) 
& = \frac{1}{1-x} C\left(\frac{x^2}{(1-x)^3}\right)
= \sum_{m \ge 0} \frac{1}{m+1}\binom{2m}{m}\frac{x^{2m}}{(1-x)^{3m+1}}%\\
%&= \sum_{m \ge 0} \frac{1}{m+1}\binom{2m}{m} \left(\sum_{n\ge0}\binom{n+m}{3m}x^n \right)
\end{align*}
and the two identities 
\begin{align*}
F_D\one(x) &= x F\one(x) \quad \text{ and } \quad
F_E\one(x) = (1-x) F\one(x) - 1,
\end{align*}
completing the proof.
\end{proof}

The concept of \emph{symmetric peaks} was presented by Fl\'{o}rez and Rodr\'{i}guez \cite{FR20} as follows.
Every peak can be extended to a unique maximal subsequence
of the form $N^{i} E^{j}$ for $i$, $j \ge 1$,
which we call the \emph{maximal mountain} of the peak.
A peak is \emph{symmetric} if its maximal mountain $N^{i} E^{j}$ satisfies $i = j$, and it is \emph{asymmetric} otherwise.
Elizalde \cite[Theorem 2.1]{Eli21} found the generating function $C_{sp,ap}(t,r,z)$ for Dyck paths with respect to the number of symmetric peaks and asymmetric peaks.
We found that the equation 
$$C_{sp,ap}(0,1,x) = 1 + x F_E\one(x)$$
holds true.

%For $F_E\one(x)$, we have
%$$[x^n]F_E\one(x)=[z^{n+1}]C_{sp,ap}(0,1,z),$$
%where $C_{sp,ap}(t,r,z)$ is the expression in \cite[Theorem 2.1]{Eli21}.
%%This sequence is given by entry A188464 in the OEIS \cite{Slo18}.

\begin{conj}
\label{conj:1}
For a positive integer $n$,
there exists a bijection between the following two sets:
\begin{enumerate}[(i)]
\item the set $E\P\one_{n, n}(NE, EN)$ of Delannoy paths from $(0,0)$ to $(n, n)$
consisting of north-steps $N=(0,1)$, east-steps $E=(1,0)$, and diagonal-steps $D=(1,1)$
that 
avoid patterns $NE$ and $EN$, 
end with east step $E$, 
and are not below $y=x$ 
\item the set of Catalan paths of size $n+1$ avoiding symmetric peaks.
\end{enumerate}
\end{conj}

%%%%%%%%%%%%%%%%%%%%%%%%%%%%%%%%%%%%%%%%%%%%%%%%%%%%%%%%%%%%%%%%%%%%%%%%%%%%%%
\subsection*{Enumerations of $\P\two_{n,2n}(D,EENN)$ and $\P\two_{n,2n}(NE,EN)$}

For $k=2$,
from the next theorem, %~\ref{thm:k=2},
the series $F$, $F_D$, and $F_E$ begins with
\begin{align*}
F\two(x) &= 1 + x + 3x^2 + 11x^3 + 44 x^4 + 186 x^5 + 818 x^6 + 3706 x^7 + \cdots,\\
F_D\two(x) &= x + 2x^2 + 6x^3 + 22x^4 + 89x^5 + 381 x^6 + 1694 x^7 + \cdots,\\
F_E\two(x) &= x^2 + 5x^3 + 22x^4 + 97 x^5 + 437x^6 + 2012 x^7 +\cdots,
\end{align*}
where the sequence of the coefficients in $F_D\two(x)$ is given by entry A200753 in the OEIS \cite{Slo18}.
However, the sequence
$$1, 1, 3, 11, 44, 186, 818, 3706, 17182, 81136, \ldots$$
of the coefficients in $F\two(x)$ is not in the OEIS \cite{Slo18}.

%$F, F_D$의 $x^n$의 계수는 다음과 같다.
\begin{thm}% [k=2]
\label{thm:k=2}
For $k=2$, the coefficients of $x^n$ of $F$ and $F_D$ are as follows:
\begin{align*}
[x^n] F\two(x) &= \sum_{m\ge0} \frac{ (-1)^{n-m}}{m+1} {3m+1 \choose m}{m+1 \choose n-m}, \\
[x^n] F_D\two(x) &= \sum_{m\ge1} \frac{ (-1)^{n-m}}{m} {3m \choose m-1}{m \choose n-m}.
\end{align*}
\end{thm}
\begin{proof}
Setting $k=2$ in \eqref{eq:F_D}, we obtain
\begin{align*}
F_D\two(x) &= (x-x^2)(1+F_D\two(x))^3
\end{align*}
and, letting $t=x-x^2$,
\begin{align*}
A(t) = t(1+A(t))^3,
\end{align*}
where the power series $A(t)$ satisfies $A(x-x^2)=F_D\two(x)$.
Setting $k=2$ in \eqref{eq:decom2}, we obtain
%$$
%F_D = xF(1+F_D).
%$$
$$
F_D\two(x) = \frac{xF\two(x)}{1-xF\two(x)}.
$$
Substituting $F_D\two(x)$ in \eqref{eq:F_D} with it, we obtain
\begin{align*}
xF\two(x)=(x-x^2) \left(\frac{1}{1-xF\two(x)}\right)^2
\end{align*}
and, letting $t=x-x^2$,
\begin{align*}
B(t) = t\left(\frac{1}{1-B(t)}\right)^2,
\end{align*}
where the power series $B(t)$ satisfies $B(x-x^2)=xF\two(x)$.

Using the Lagrange inversion formula, we obtain the coefficients of $t^n$ in $A(t)$ and $B(t)$.
%\seo{계산과정 상술하기 $k=1$ 때처럼}
Replacing $t$ with $x-x^2$, we obtain the coefficients of $x^n$ in $F_D\two(x)$ and $xF\two(x)$.
\end{proof}

The function $F_D\two(x)$ is mentioned in \cite[Theorem 3.7]{MS15}.
It is shown here that $F_D\two(x)$ is the same as the generating function of $\mathcal{IS}_n(102)$.
In the remark of \cite[p.168]{MS15}, they were unable to find other combinatorial interpretations in the literature.

\begin{conj}
\label{conj:2}
For a positive integer $n$,
there exists a bijection between the following two sets:
\begin{enumerate}[(i)]
%\item the set $\P\two_{n-1, 2n-1}(NE, EN)$ of Delannoy paths from $(0,0)$ to $(n-1, 2n-1)$
%consisting of north-steps $N=(0,1)$, east-steps $E=(1,0)$, and diagonal-steps $D=(1,1)$ avoiding the patterns $NE$ and $EN$, that do not go below $y=2x$ and
\item the set $D\P\two_{n, 2n}(NE, EN)$ of Delannoy paths from $(0,0)$ to $(n, 2n)$
consisting of north-steps $N=(0,1)$, east-steps $E=(1,0)$, and diagonal-steps $D=(1,1)$ 
that 
avoid the patterns $NE$ and $EN$, 
end with diagonal step $D$, 
and are not below $y=2x$ and
\item the set $\mathcal{IS}_n(102)$ of inversion sequences of length $n$ avoiding the pattern $102$.
\end{enumerate}
\end{conj}

%%%%%%%%%%%%%%%%%%%%%%%%%%%%%%%%%%%%%%%%%%%%%%%%%%%%%%%%%%%%%%%%%%%%%%%%%%%%%%
\subsection*{Enumerations of $\P\three_{n,3n}(D,EENN)$ and $\P\three_{n,3n}(NE,EN)$}
From Equation~\eqref{eq:F_D} for $k=3$,
we obtain
\begin{align*}
F_D &= x(1+F_D)^{4} - x^2(1+F_D)^{5}. %\label{eq:F_D}
\end{align*}

According to a computer program, we find that the series $F$, $F_D$, and $F_E$ begins with
\begin{align*}
F\three(x) &= 1+x+4 x^{2}+20 x^{3}+111 x^{4}+657 x^{5}+4065 x^{6} + 25981 x^{7}+ \cdots, \\
F_D\three(x) &= x+3 x^{2}+13 x^{3}+67 x^{4}+380 x^{5}+2288 x^{6}+ 14351 x^{7}+ \cdots, \\
F_E\three(x) &= x^{2}+7 x^{3}+44 x^{4}+277 x^{5}+1777 x^{6} +11630 x^{7} + \cdots,
\end{align*}
where the sequence of the coefficients in $F_D\three(x)$ is given by entry A200754 in the OEIS \cite{Slo18}.

%For arbitrary positive integer $k$, the following relation holds between $F$ and $F_D$
%$$
%F^2-(1+F_D)^2 F + F_D(1+F_D) =0.
%$$
Note that, for $k \ge 3$, we were unable to find a solution for \eqref{eq:F_D}.

\section*{Acknowledgements}
%The authors thank Professor Jong Young Lee for drawing our attention to this problem.
For the first author, this work was partially supported by the 2018 Research Grant from Kangwon National University(No. 520180091) and the National Research Foundation of Korea (NRF) grant funded by the Korea government (MSIT) (No. 2020R1F1A1A01065817).
For the second author, this work was supported by the National Research Foundation of Korea (NRF) grant funded by the Korea government (MSIT) (No. 2017R1C1B2008269).

%--------------------------------------------------------------------------------
%\nocite{*}
%%\bibliographystyle{abbrvnat}
%% use the following instead if you encounter problems

%\bibliographystyle{alpha}
%%\bibliographystyle{amsabbrv}
%\bibliography{\jobname}

\end{document}